\documentclass[12pt,a4paper]{amsart}
\usepackage[latin1]{inputenc}
\usepackage{amsmath}
\usepackage{amsfonts}
\usepackage{amssymb}
\usepackage{amsthm}
\usepackage{xfrac}
\usepackage{graphicx}
\usepackage{enumitem}

\newtheorem{theorem}{Theorem}
\newtheorem{prop}{Proposition}[section]
\newtheorem*{theorem*}{Theorem}
\newtheorem{lem}[prop]{Lemma}
\theoremstyle{definition}
\newtheorem{definition}[prop]{Definition}

\theoremstyle{remark}
\newtheorem{remark}[prop]{Remark}
\newtheorem*{remark*}{Remark}
\newtheorem{corollary}[prop]{Corollary}

\numberwithin{equation}{section}

\author[J.~Breuer, S.~Denisov, and L.~Eliaz]{Jonathan~Breuer$^{1,3}$, Sergey Denisov$^{2}$ and Latif Eliaz$^{1,4}$}
\title[The Essential Spectrum on Trees]{On the Essential Spectrum of Schr\"odinger Operators on Trees}
\begin{document}
\sloppy
\thanks{$^1$ Institute of Mathematics, The Hebrew University of Jerusalem, Jerusalem, 91904, Israel.
Supported in part by the Israel Science Foundation (Grant No.\ 399/16) and in part by the United States-Israel Binational Science Foundation
(Grant No.\ 2014337)}

\thanks{$^2$ Department of Mathematics, University of Wisconsin-Madison, Madison WI 53706-1388, USA. Work on
 section 4 was supported by grant RSF-14-21-00025 and research conducted on other sections was supported by grant NSF-DMS-1464479 and by Van Vleck Professorship Research Award. Email: denissov@wisc.edu}

\thanks{$^3$ E-mail: jbreuer@math.huji.ac.il}

\thanks{$^4$ E-mail: latif.eliaz@mail.huji.ac.il}
\maketitle

\begin{abstract}
It is known that the essential spectrum of a Schr\"odinger operator $H$ on $\ell^{2}\left(\mathbb{N}\right)$ is equal to the union of the
spectra of right limits of $H$.
The natural generalization of this relation to $\mathbb{Z}^{n}$ is known to hold as well.

In this paper we generalize the notion of right limits to general infinite connected graphs and construct examples of graphs for which the essential
spectrum of the Laplacian is strictly bigger than the union of the spectra of its right limits. As these right limits are trees, this
result is complemented by the fact that the equality still holds for general bounded operators on regular trees. We prove this and
characterize the essential spectrum in the spherically symmetric case.

\end{abstract}
\section{Introduction}
Let $G$ be a graph with vertices $V\left(G\right)$ and edges $E\left(G\right)$. The degree of a vertex $v \in V(G)$, $d(v)$, is the number of
$u \in V(G)$ such that $u \sim v$, where we denote $u\sim v$ for vertices $u,v\in G$ if $\left(u,v\right)\in E\left(G\right)$. $G$ is called
regular if all vertices have the same degree. A connected graph which has no cycles is called a tree.

A Jacobi operator on a graph $G$ is an operator, $H$, acting on $\psi\in\ell^2\left(V\left(G\right)\right)\cong\ell^2\left(G\right)$ by
\begin{equation} \nonumber
\left(H\psi \right)\left( v \right)=\sum_{u\sim v}a_{u,v}\psi(u)+b \left(v \right)\psi \left(v \right)
\end{equation}
where $a: \left\{(u,v) \mid u\sim v \right\} \rightarrow (0,\infty)$, satisfying $a_{u,v}=a_{v,u}$, and $b: V\left(G\right) \rightarrow \mathbb{R}$ are functions (which we
take to be bounded throughout the paper). The particular case of $a \equiv 1$ and $b(v)=-\textrm{degree}(v)$ is the discrete graph Laplacian,
$\Delta$. The canonical example we have in mind is that of a Schr\"odinger operator, namely $H=\Delta+Q$, where $Q$ is the operator of
multiplication by a bounded real valued function.

Given a bounded self-adjoint operator, $A$, let $\sigma(A)$ denote its spectrum and $\sigma_\text{discrete}\left(A\right)$ denote the set of
isolated eigenvalues of $A$ of finite multiplicity (=the discrete spectrum). The essential spectrum of $A$ is the set
$\sigma_\text{ess}\left(A\right)=\sigma\left(A\right)\backslash\sigma_\text{discrete}\left(A\right)$.

The essential spectrum of a self-adjoint operator can be characterized as that part of the spectrum that is invariant under compact
perturbations (this is the content of Weyl's Theorem \cite[Theorem S.13]{ReedSimonI}) and so, for a Schr\"odinger operator defined over an
infinite graph, should morally depend only on properties of the operator `at infinity'. For the case of $\mathbb{Z}^d$, this intuition was
made precise by Last-Simon in \cite{LastSimonEss} using the concept of `right limit' (introduced in \cite{LastSimonAC}). The purpose of this
paper is to discuss the possibility of extending this idea to general connected graphs.

For the purpose of this introduction, we define right limits on $\mathbb{N}$. Consider a bounded Jacobi matrix $H$ acting
on $\ell^{2}\left(\mathbb{N}\right)$. A Jacobi matrix $H^{\left(r\right)}$,
acting on $\ell^{2}\left(\mathbb{Z}\right)$, is a right limit of
$H$ if there exists a sequence of indices $\left\{ n_{j}\right\}_{j=1}^\infty \subseteq\mathbb{N}$
s.t.\ for every fixed $l\in\mathbb{Z}$,
\begin{equation} \label{eq:rightlimdef}
a_{l+n_j,l+1+n_j}\stackrel{j\to\infty}{\longrightarrow}a^{(r)}_{l,l+1}, \quad b_{l+n_j}\stackrel{j\to\infty}{\longrightarrow}b^{(r)}_{l}.
\end{equation}
Equivalently, if we extend $H$ to an operator $\tilde{H}$ on $\ell^2\left(\mathbb{Z}\right)$, we will get $H^{(r)}$ as a strong limit\footnote{We recall that a sequence of operators, $\{A_n\}_n$ defined over a Hilbert space $\mathcal{H}$, is said to converge strongly to an operator $A$ if for any $\psi \in \mathcal{H}$, $\lim_{n \rightarrow \infty} A_n \psi=A\psi$. Similarly, it is said to converge weakly if for any $\psi, \phi \in \mathcal{H}$, $\lim_{n \rightarrow \infty} \left(\phi, A_n \psi \right)=\left(\phi, A\psi \right)$.} of a
sequence of left-shifts of $\tilde{H}$ (corresponding to the sequence $\left\{ n_{j}\right\}$). Note that, by compactness, one can always find
a right limit along a subsequence of any such sequence of shifts.

Thus, a right limit of $H$ is a limit point of shifts of $H$. The concept of right limit was extended in \cite{LastSimonEss} to operators on
$\mathbb{Z}^d$, where in this case all possible directions towards infinity must be considered. We shall extend the notion of right limit to
general connected graphs in Section \ref{sec:rlimdef} below. Since the word `right' no longer makes sense in this context, we shall name the relevant
objects $\mathcal{R}$-limits.

The following characterization of the essential spectrum of Jacobi matrices is essentially from \cite{LastSimonEss} (for related results see
\cite{AmreinMP, Anselone, CWL, GI1,GI2,GI3,Kurbatov, LanRab, M1, MPR,Muham1,Muham2, RRS,RRSband, SeidelSil, Shubin1,Shubin2,SimonSz};
comprehensive reviews and further references on the subject can be found in \cite{CWL, LastSimonEss, SimonSz}).

\begin{theorem}[\cite{LastSimonEss,RRS,SimonSz}] \label{thm:LastSimon}
 Assume $H$ is a bounded Jacobi matrix on $\ell^{2}\left(\mathbb{N}\right)$ or on $\ell^2 \left(\mathbb{Z}^d \right)$. Then
\begin{eqnarray} \label{eq:Last-Simon}
\sigma_{ess}(H) & = & \bigcup_{H^{\left(r\right)}\text{ is a right limit of
\ensuremath{H}}}\sigma\left(H^{\left(r\right)}\right)\label{eq:mainResult}
\end{eqnarray}
\end{theorem}
\begin{remark*}
This theorem was stated in \cite{LastSimonEss} with $\overline{\bigcup_{r}\sigma\left(H^{\left(r\right)}\right)}$ on the right hand side.
However, $\bigcup_{r}\sigma\left(H^{\left(r\right)}\right)$ is in fact closed (see e.g.\ \cite{RRS} and \cite{SimonSz} for details).
\end{remark*}

In the context of regular trees, Golenia \cite{Golenia} and Golenia-Georgescu \cite{GolGeor}, have shown (\ref{eq:mainResult}) in the
particular case of Schr\"odinger operators when $Q$ has a limit (in the usual sense) on every path to infinity. We are not aware of other
results treating general potentials $Q$ on graphs and trees.

One direction of the inclusion in \eqref{eq:Last-Simon} is almost trivial for operators on $\ell^2 \left(\mathbb{Z}^d \right)$ and follows
almost immediately from the characterization of the essential spectrum via an orthogonal sequence of approximate eigenfunctions. This is true
for the case of general graphs (with the same argument). Nevertheless, we give a proof of this theorem in Section \ref{sec:thm1} below for
completeness.

\begin{theorem}\label{thm:Thm1}Assume $H$ is a bounded Jacobi matrix on
$\ell^{2}\left(G\right)$ where $G$ is a connected graph of bounded degree, then,
\[
\bigcup_{L\text{ is an \ensuremath{\mathcal{R}}-limit of \ensuremath{H}}}\sigma\left(L\right)\subseteq\sigma_{\text{ess}}\left(H\right)
\]
\end{theorem}

We remark that we expect the analogue of this theorem to hold for the generalization of Jacobi matrices to band-dominated operators (see e.g.\
\cite{RRS} for the definition).

Our first main result is the somewhat surprising fact that for general graphs the reverse inclusion fails.

\begin{theorem}
\label{thm:Thm2}There exists a connected graph $G$ s.t.\ the adjacency operator (i.e., the Jacobi matrix with $a \equiv 1$, $b \equiv 0$) on $G$
satisfies,
\[
\sigma_{\text{ess}}\left(A_G\right)\Bigg\backslash\overline{\bigcup_{L\text{ is an \ensuremath{\mathcal{R}}-limit of
}A_G}\sigma\left(L\right)} \neq \emptyset.
\]
\end{theorem}

While the graph of Theorem \ref{thm:Thm2} is not a tree, its construction involves the use of a sequence of regular graphs with girth growing
to infinity. It has three right limits, a line, a $d$-regular tree, and a gluing of the two. It is thus of some interest that the result still
holds on regular trees.

\begin{theorem}
\label{thm:Thm3}Assume $H$ is a bounded Jacobi matrix on $\ell^{2}\left(T\right)$ where $T$ is a regular tree, then
\begin{equation} \label{eq:closureEss}
\sigma_{\text{ess}}\left(H\right)=\overline{\bigcup_{L\text{ is an \ensuremath{\mathcal{R}}-limit of \ensuremath{H}}}\sigma\left(L\right)}.
\end{equation}
\end{theorem}

\paragraph*{Remarks}
\begin{enumerate}

\item The proof of this theorem is an adaptation of the method of proof from \cite{LastSimonEss} to the case of a tree, and relies on the
    construction of an appropriate partition of unity. However, the direct approach of defining the relevant functions on balls fails due to
    the fact that the size of the boundary of balls is comparable to their volume. Nevertheless, by restricting attention to annuli and
    using the structure of the tree, it is possible to overcome this obstacle.

\item Note the closure in \eqref{eq:closureEss}. Since our proof is an adaptation of \cite{LastSimonEss}, we cannot do better with this
    method. We believe it is possible to show that the set on the RHS is closed, however this seems to require some delicate estimates on
    the growth of generalized eigenfunctions on trees and is likely to be somewhat technical. Since this is not in the main thrust of this
    paper, we postpone this analysis to a later work. Nevertheless, we note that our discussion in Section \ref{sec:spherical} shows that in
    the spherically symmetric case \eqref{eq:closureEss} holds without the closure (see \eqref{eq:EssNoClosure}).

\item We expect Theorem \ref{thm:Thm3}, and its proof, to carry over to the case of non-regular trees as well. We restrict our attention to
    regular trees here for simplicity.

\item It is interesting to study the possibility to generalize another characterization of the essential spectrum of Jacobi matrices which
    holds in the 1-dimensional case. In this case, the essential spectrum is characterized in terms of
    $\sigma_{\infty}\left(J^{\left(r\right)}\right)$, the pure point spectra in $\ell^{\infty}\left(\mathbb{Z}\right)$ of the right limits
    (see e.g.\  \cite[Theorem 7.2.1]{SimonSz}).

\end{enumerate}

The structure of this paper is as follows. First, in Section \ref{sec:rlimdef} we define $\mathcal{R}$-limits on infinite graphs. Sections
\ref{sec:thm1}, \ref{sec:example} and \ref{sec:thm3} are dedicated to the proofs of Theorems \ref{thm:Thm1}, \ref{thm:Thm2} and \ref{thm:Thm3}
respectively. Finally, in Section \ref{sec:spherical} we discuss the case of spherically symmetric operators on a regular tree, where the
essential spectrum can be characterized using right limits and truncations of right limits of a single one-dimensional Jacobi matrix.


\section{$\mathcal{R}$-Limits of General Graphs} \label{sec:rlimdef}

This section deals with the extension of the concept of right limits to general connected graphs with bounded degree. There are two issues that make the analogous notion of right limit for general graphs more complex than that of the one-dimensional object. The first (minor) one is the fact that general graphs may have multiple paths to infinity. This is true already in the case of $\mathbb{Z}^d$ and is the main reason why we
refrain from using the name `right limit' in this case and use $\mathcal{R}$-limit instead. The second issue is that with a general graph the
absence of homogeneity means that the different $\mathcal{R}$-limits of an operator might be defined on various different graphs which are not
necessarily related in a simple way to the graph over which the original operator was defined. Thus, one is faced with the requirement to
compare operators defined over different graphs. In order to deal with the first issue, one has to specify a path to infinity (in fact, it suffices to consider sequences of vertices with increasing distance from some fixed root, but considering paths is equivalent and more convenient). In order to
deal with the second one, we need to introduce local mappings to finite dimensional vector spaces which will satisfy a certain compatibility
condition with each other.

Let $H$ be a Jacobi matrix on a connected graph $G$ with bounded degree. For any vertex $v\in G$ and $r\in\mathbb{N}$ denote the ball
\[
B_r\left(v\right)=\left\{u\in G\,|\,\text{dist}(u,v)\leq r\right\},
\]
and denote by $N_{v,r}$ the number of vertices in this ball, i.e.\ $N_{v,r}=\left|B_r\left(v\right)\right|$. Let $H_r^{(v)}=H|_{B_r(v)}$(=the
restriction of $H$ to $\ell^2 \left(B_r(v) \right)$).

Let $\eta$ be an indexing of the vertices of this ball,
\begin{equation} \nonumber
\eta:B_r(v)\to\left\{1,2,\ldots,N_{v,r}\right\}
\end{equation}
and define the corresponding unitary mapping $\mathcal{I}_\eta:\ell^2\left(B_r\left(v\right)\right)\to \mathbb{C}^{N_{v,r}}$ by
\begin{equation} \nonumber
\mathcal{I}_\eta \left(\delta_u\right)=e_{\eta(u)},
\end{equation}
for any vertex $u\in B_r(v)$, where $\delta_u$ is the delta function at $u$ and $\left \{e_1,e_2,\ldots,e_{N_{v,r}} \right \}$ is the standard
basis in $\mathbb{C}^{N_{v,r}}$ (i.e.\ $e_i\left(j\right)=\delta_{i,j}$).
Let $M^{(v)}_{\eta, r}\in\mathcal{M}_{N_{v,r},N_{v,r}}$ be the matrix defined by
\begin{equation} \nonumber
M^{(v)}_{\eta,r}=\mathcal{I}_\eta H_r^{(v)} \mathcal{I}_\eta^{-1}.
\end{equation}

\begin{definition}
Fix a vertex $v\in G$ and for any $r \in \mathbb{N}$, let
\begin{equation} \nonumber
\eta_r :B_r(v)\to\left\{1,2,\ldots,N_{v,r}\right\}
\end{equation}
be an enumeration as above, and $\mathcal{I}_r=\mathcal{I}_{\eta_r}$ be the corresponding isomorphism.
We say that the sequence of isomorphisms $\left\{\mathcal{I}_r \right\}_{r=1}^\infty$ is \emph{coherent} if for any $r<s$ and any $u \in
B_r(v)$
\begin{equation} \nonumber
\eta_s(u)=\eta_r(u).
\end{equation}
When we want to emphasize the dependence on $v$, we say that  $\left\{\mathcal{I}_r \right\}_{r=1}^\infty$ is a coherent sequence at $v$.
\end{definition}

Note that, if $\{\mathcal{I}_r\}_{r=1}^\infty$ is a coherent sequence of isomorphisms at $v \in G$, then for any $r$, the corresponding matrix
$M^{(v)}_{\eta_r,r}$ is the $N_{v,r} \times N_{v,r}$ upper left corner of the matrix $M^{(v)}_{\eta_s,s}$ for any $r \leq s$. Thus, in what
follows, when the coherent sequence is clear, we omit the $\eta_r$ and write simply $M^{(v)}_r=M^{(v)}_{\eta_r,r}$.

We say that a sequence of vertices $\left\{v_n\right\}_{n=0}^\infty$ is a path to infinity in $G$ if $v_{n+1}\sim v_n\ \forall
n\in\mathbb{N}$, and
$|v_n|=\text{dist}\left(v_n,v_0\right)\raisebox{\dimexpr-0.2\height}{\rotatebox{0}{$\xrightarrow{\makebox[0.5cm]{$\rotatebox{-0}{$\scriptstyle{n\to\infty}$}$}}$}}\infty$
monotonically.
\begin{definition} \label{def:rlimdef}
Given a graph $G'$, a vertex $v_0'\in G'$ and a Jacobi matrix $H'$ on $G'$, we say that $\left\{H',G',v_0'\right\}$ is an $\mathcal{R}$-limit
of $H$ along the path to infinity $\left\{v_n\right\}_{n=0}^\infty$ if there exists a sequence of indices $\left\{n_j\right\}_{j=1}^\infty$,
such that
\begin{enumerate}[leftmargin=*]
\item[(i)]For any $j\in\mathbb{N}$ there exists a coherent sequence of isomorphisms $\left\{\mathcal{I}^{(j)}_{k}\right\}_{k=1}^\infty$ at
    $v_{n_j}$.
\item[(ii)]There exists a coherent sequence of isomorphisms $\left\{\mathcal{I}'_{k}\right\}_{k=1}^\infty$ at ${v_0'}$.
\item[(iii)]For any $r\in\mathbb{N}$ $N_{v_{n_j},r}=N_{v_0',r}$ for all sufficiently large $j$, and
\begin{equation} \label{eq:rlimdef}
\lim_{j\to\infty}M^{\left(v_{n_j}\right)}_r=M^{(v_0')}_{r}.
\end{equation}
\end{enumerate}
\end{definition}

In the one dimensional case, the matrices $M^{(v_j)}_r$ are simply truncated Jacobi matrices and \eqref{eq:rlimdef} translates to the
condition \eqref{eq:rightlimdef}. Thus, the definition of $\mathcal{R}$-limits is a direct generalization of the definition of right limits in
the one dimensional case.


\section{Proof of Theorem 2} \label{sec:thm1}
First, recall Weyl's criterion (see, e.g.\ \cite[Theorem VII.12]{ReedSimonI}) for the essential spectrum of a bounded self adjoint operator
$A$ on a Hilbert space $\mathcal{H}$: ${\lambda\in\sigma_{\text{ess}}\left(A\right)}$ iff there exists an orthonormal sequence
$\left\{\psi_n\right\}_{n=1}^\infty\subset\mathcal{H}$ of approximate eigenfunctions, i.e.\ functions satisfying
$\left(\psi_n,\psi_k\right)=\delta_{n,k}$ and \[\left\Vert\left(A-\lambda\right)\psi_n\right\Vert\xrightarrow{n\to\infty}0.\]

\begin{proof} [Proof of Theorem \ref{thm:Thm1}]
Assume $\{H',G',v_0'\}$ is an $\mathcal{R}$-limit of $H$ along a path to infinity $\left\{v_j\right\}_{j=0}^\infty$, and
$\lambda\in\sigma\left(H'\right)$.
Given $\varepsilon>0$ define
\[
\varepsilon'=\min\left(\frac{2\varepsilon}{1+\Vert H'\Vert+|\lambda|},\frac{1}{2}\ \right).
\]
Since $\lambda\in\sigma(H')$ there exists $\psi\in\ell^{2}\left(G'\right)$
s.t.\ $\left\Vert \left(H'-\lambda\right)\psi\right\Vert <\varepsilon'$ and $\left\Vert \psi\right\Vert =1$.
Additionally, since $\psi\in\ell^2\left(G'\right)$, there exists $R>0$ s.t.\
\[
\left\Vert \psi|_{G'\backslash B_R(v_0')}\right\Vert <\varepsilon'.
\]
Thus by defining for every $w\in V(G')$
\[
\varphi(w)=\begin{cases}
\psi(w)/{K} & w\in B_R(v_0') \\
0 & \text{otherwise},
\end{cases}
\]
with
\[
K=\left\Vert \psi|_{B_R(v_0')}\right\Vert>\frac{1}{2},
\]
we get an approximate eigenfunction for $H'$, supported on $B_R\left(v_0'\right)$, and satisfying $\left\Vert \varphi\right\Vert =1$. Indeed
\[
\left\Vert\left(H'-\lambda\right)\varphi\right\Vert=
\left\Vert\left(H'-\lambda\right)\frac{\psi-(\psi-\varphi)}{K}\right\Vert< {{1+\left\Vert H'\right\Vert+\left|\lambda\right|}\over{K}}
\varepsilon'\leq\varepsilon.
\]
Since $H'$ is an $\mathcal{R}$-limit of $H$ there exists some $u=v_{n_j}\in G$ s.t.\ the corresponding matrices satisfy
\[
\left\Vert M^{(u)}_{R+2} - M^{(v_0')}_{R+2} \right\Vert < \varepsilon.
\]
Let $\mathcal{I}:\ell^2\left(B_{R+2}(u)\right)\to\mathbb{C}^{N_{u,R+2}}$ and
$\mathcal{I'}:\ell^2\left(B_{R+2}(v_0')\right)\to\mathbb{C}^{N_{v_0',R+2}}$ be the isomorphisms from Definition \ref{def:rlimdef}.
Denote by $\chi'$ the function $\chi'=\mathcal{I}^{-1}\mathcal{I}'\widetilde{\varphi} \in\ell^2\left(B_{R+2}(u)\right)$,
where $\widetilde{\varphi}=\varphi|_{B_{R+2}(v'_0)}$. Additionally define
\[
\chi(w)=\begin{cases}
\chi'(w) & w\in B_{R+2}\left(u\right)\\
0 & \text{otherwise}.
\end{cases}
\]
Then
\[
\left(H-\lambda\right)\chi=\left(H|_{B_{R+2}(u)}-\lambda\right)\chi'
\]
and thus,
\begin{eqnarray}
&\left\Vert \left(H-\lambda\right)\chi\right\Vert =
\left\Vert \left(M^{(u)}_{R+2}-\lambda\right)\mathcal{I}'\widetilde{\varphi}\right\Vert < \nonumber\\
&\left\Vert \left(M^{(u)}_{R+2}-M^{(v_0')}_{R+2}\right)\mathcal{I}'\widetilde{\varphi}\right\Vert +
\left\Vert \left(M^{(v_0')}_{R+2}-\lambda\right)\mathcal{I}'\widetilde{\varphi}\right\Vert <\nonumber\\
&\varepsilon + \left\Vert \left(H'-\lambda\right)\varphi\right\Vert < 2\varepsilon \nonumber
\end{eqnarray}
We can now repeat this argument for a subsequence of vertices along the sequence $\left\{v_{n_j}\right\}$ from Definition \ref{def:rlimdef},
s.t.\ $\text{dist}\left(u_1,u_2\right)>R+2$ for any two vertices $u_1, u_2$ on this subsequence.
As a result we get for any $\varepsilon>0$ an orthonormal sequence of (compactly supported) functions $\left\{\varphi_k\right\}_{k=1}^\infty$
satisfying,
\[
\left\Vert \left(H-\lambda\right)\varphi_k\right\Vert <\varepsilon.
\]
Thus, by taking e.g.\ $\varepsilon_n=\frac{1}{n}$, we can choose an orthonormal sequence of approximate eigenfunctions for $H$.
Hence by Weyl's criterion $\lambda\in\sigma_\text{ess}\left(H\right)$.
\end{proof}


\section{Proof of Theorem 3} \label{sec:example}

\begin{proof}[Proof of Theorem \ref{thm:Thm2}]
We shall construct a graph $G$ for which the adjacency operator $H=A_G$ on $\ell^{2}\left(G\right)$
satisfies
\[
\sigma_{\text{ess}}\left(A_G\right)\Bigg\backslash\overline{\bigcup_{L\text{ is an \ensuremath{\mathcal{R}}-limit of }
A_G}\sigma\left(L\right)}
\]
is nonempty.

We recall that the girth of a graph $\mathcal{G}$ is
\[
\text{girth}\left(\mathcal{G}\right)\equiv\min\left\{\text{length}(l)\,|\,l\text{ is a cycle in }\mathcal{G}\right\}.
\]
Fix $d>2$ and let $\left\{ G_{n_{i}},u^{(1)}_i,u^{(2)}_i \right\} _{i=1}^{\infty}$
be a sequence of $d$-regular graphs on $n_{i}$ vertices, each with two marked vertices $u^{(1)}_i,u^{(2)}_i \in G_{n_i}$, where,
$\left\{n_{i}\right\}_{i=1}^{\infty}\subset\mathbb{N}$, $n_{i}\to\infty$, and s.t.\
\begin{align*}
&\text{girth}\left(G_{n_{i}}\right)\overset{i}{\longrightarrow}\infty,\\
&\text{dist}\left(u^{(1)}_i,u^{(2)}_i\right)\overset{i}{\longrightarrow}\infty.
\end{align*}
By, e.g., \cite{LPS} such a sequence exists for $d=p+1$ for any prime $p\neq1$ satisfying $p\equiv1\left(\text{mod}\,4\right)$.
Additionally, let $\left\{ k_{i}\right\} _{i=1}^{\infty}\subset\mathbb{N}$, be an increasing sequence s.t.\ $k_{i+1}-k_{i}\to\infty$. We
construct $G$ by `replacing' the edge $(k_i, k_i+1)$ in $\mathbb{N}$ by the graph $G_{n_{i}}$. This is done by cutting $(k_i, k_i+1)$ and
attaching $u^{(1)}_i$ to $k_i$ and $u^{(2)}_i$ to $k_{i+1}$ (see Figure \ref{figure1}). Formally
\begin{eqnarray*}
V(G)=&\mathbb{N}\cup \bigcup_{i=1}^\infty V\left(G_{n_i}\right),\\
E(G)=&\left(\bigcup_{i=1}^\infty E\left(G_{n_i}\right)\right) \cup \left(E\left(\mathbb{N}\right)\Big\backslash \cup_{i=1}^\infty
\left\{(k_i,k_{i+1})\right\}\right)\cup \\
&\left(\bigcup_{i=1}^\infty \left\{(k_i,u^{(1)}_i)\right\}\right)\cup \left(\bigcup_{i=1}^\infty \left\{(u^{(2)}_i,k_{i+1})\right\}\right).
\end{eqnarray*}

\begin{figure}[ht!]
\centering
\includegraphics[width=135mm]{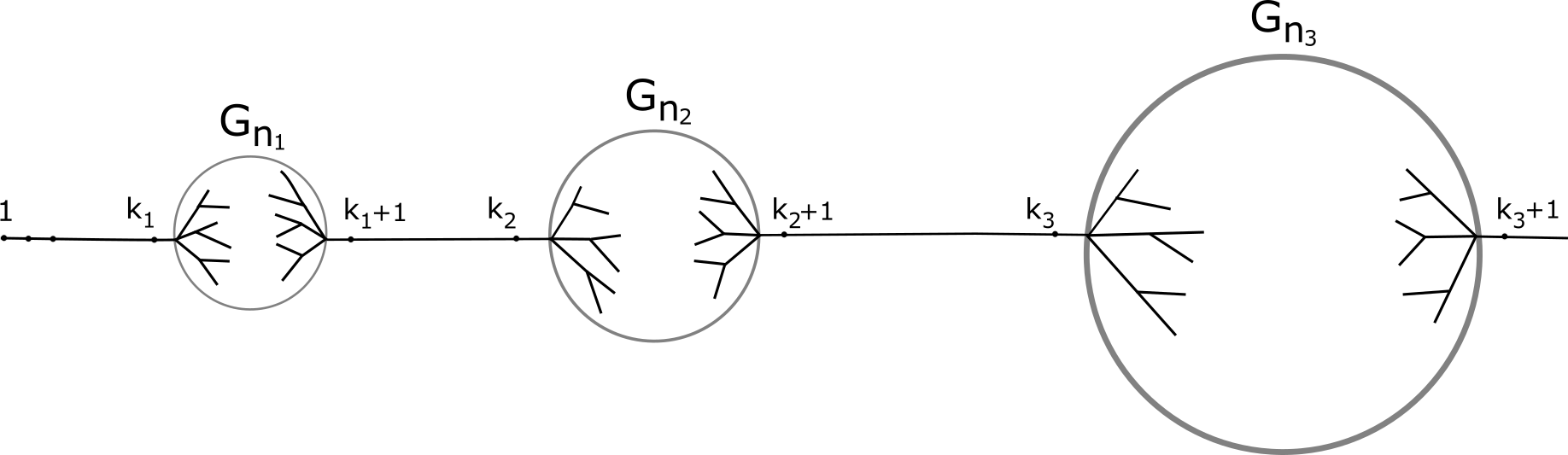}\\
\caption{The construction of the graph $G$ for the counterexample. \label{figure1}}
\end{figure}

Let $H=A_G$ on $G$, i.e.\ the potential is $Q(v)=\text{deg}\left(v\right)$. For each graph
$G_{n_i}$ the constant function $\varphi\left(v\right)=\frac{1}{\sqrt{n_i}}$ is an eigenfunction of $H$ with eigenvalue $\lambda=d$. Define
\[
\varphi_{i}\left(v\right)=\begin{cases}
\sfrac{1}{\sqrt{n_{i}}}\,\,\,\,\, & \text{if }v\in G_{n_{i}}\\
0 & \text{otherwise.}
\end{cases}
\]
Then, summing over the boundary terms, we have for any $i\in\mathbb{N}$
\[
\left\Vert H\varphi_{i}-\lambda\varphi_{i}\right\Vert^{2}=\sfrac{2}{n_{i}}.
\]
Thus
\[
\left\Vert H\varphi_{i}-\lambda\varphi_{i}\right\Vert ^{2}\overset{i\to\infty}{\,\longrightarrow\,\,\,}0.
\]
Additionally, for any $i\neq j$ the functions $\varphi_{i}$ and
$\varphi_{j}$ are orthogonal. Thus $\left\{ \varphi_{i}\right\} _{i=1}^{\infty}$
is an orthonormal sequence of approximate eigenfunctions of $H$
for the value $\lambda=d$, and thus $d\in\sigma_{\text{ess}}\left(H\right).$
We claim that $d\notin\overline{\bigcup\sigma\left(L\right)}$. Indeed, it is easy to see that
the only $\mathcal{R}$-limits of $G$ are the following three objects:
\begin{enumerate}
\item The adjacency operator on the full line $\mathbb{Z}$, appearing when the limit is taken along a subsequence $\left\{ v_{n_j}\right\}
    _{j=1}^{\infty}$ of points (only) on $\mathbb{N}$, of increasing distance from the sequence $\left\{ k_{i}\right\} _{i=1}^{\infty}$,
    i.e.\
\[
\inf_{i}\left(\text{dist}\left(v_{n_j},k_{i}\right)\right)\underset{j}{\longrightarrow}\infty.
\]

\item  The adjacency operator on a $d$-regular tree $T_{d}$, appearing when $\left\{ v_{n_j}\right\} _{j=1}^{\infty}$
includes (only) points on $\left\{ G_{n_{i}}\right\}_{i=1}^{\infty} $, of increasing distance from both the sequences of vertices
$\left\{u^{(1)}_i \right\}_{i=1}^{\infty}$
and $\left\{u^{(2)}_i\right\}_{i=1}^{\infty}$, i.e.\
\[
\inf_{i}\text{dist}\left(v_{n_j},u^{(\ell)}_i\right)\underset{j}{\longrightarrow}\infty
\]
for both $\ell=1,2$.
Since the girth of $G_{n_{i}}$ grows to infinity, we conclude that
for any $R>0$ the reduced graph of radius $R$ around $v_{n_j}$ will
be a tree for $j$ large enough.
\item  The adjacency operator on the tree, $\widetilde{T}=\widetilde{T}_d$, which is a half-line connected to a $d$-regular tree at the
    point $1\in\mathbb{N}$,
appearing when $\left\{v_{n_j}\right\}_{j=1}^{\infty}$ are points
on $\mathbb{N}$ of fixed distance from $\left\{k_{i}\right\}_{i=1}^{\infty}$,
or when $\left\{ v_{n_j}\right\} _{j=1}^{\infty}$ are points from $\left\{ G_{n_{i}}\right\}_{i=1}^{\infty}$
and are of a fixed distance from either $\left\{u^{(1)}_i\right\}_{i=1}^{\infty}$
or $\left\{u^{(2)}_i \right\}_{i=1}^{\infty}$.
\end{enumerate}
The corresponding spectra for the first two operators are:
\begin{enumerate}
\item $\sigma\left(A_{\mathbb{Z}}\right)=\left[-2,2\right]$.
\item $\sigma\left(A_{T_{d}}\right)=\left[-2\sqrt{d-1},2\sqrt{d-1}\right]$.
\end{enumerate}
Both of them do not contain the point $\lambda=d$. The following lemma
completes the argument.
\begin{lem}\label{lem:Lemma1}
$d\notin\sigma\left(A_{\widetilde{T}}\right)$
\end{lem}
The proof of this Lemma is given in the Appendix.
Concluding, this example satisfies $d\notin\overline{\bigcup_{L}\sigma\left(L\right)}$,
while $d\in\sigma_{ess}(H)$.
\end{proof}
\begin{remark} \label{rem:Httspec}
In fact $\sigma\left(A_{\widetilde{T}}\right)=\sigma\left(A_{T_d}\right)$, but the calculation is a bit cumbersome, and since it is not
necessary here we do not go into details.
\end{remark}


\section{Proof of Theorem 4} \label{sec:thm3}

We start by reformulating a result from \cite{LastSimonEss} in our context: let $T$ be the $d$-regular tree and consider a bounded
self-adjoint operator, $A$, defined over $\ell^2(T)$. Let $\{j_\alpha \}_{\alpha=1}^\infty$ be a sequence of bounded, self-adjoint operators
on $\ell^2(T)$ such that
\begin{equation} \nonumber
\sum_\alpha j_\alpha^2=\textrm{Id}
\end{equation}
where the convergence in the sum is meant in the weak operator topology sense. The following is Proposition 2.1 from \cite{LastSimonEss}.

\begin{prop} \label{prop:LastSimon}
For any $\varphi \in \ell^2(T)$
\begin{equation} \nonumber
\sum_\alpha \|A j_\alpha \varphi \|^2 \leq 2 \| A \varphi\|^2+\left(\varphi, C \varphi \right)
\end{equation}
where
\begin{equation} \nonumber
C=-2\sum_{\alpha}[A,j_\alpha]^2.
\end{equation}
\end{prop}

We are now ready for the

\begin{proof}[Proof of Theorem \ref{thm:Thm3}]
We follow closely the reasoning of the proof of Theorem 1.7 of \cite{LastSimonEss}. Fix a vertex $\in V(T)$ and call it the origin $O$. For
any $v \in V(T)$, let $|v|={\rm dist}(v,O)$. Take $L\in \mathbb{N}$ and for $\alpha\in \mathbb{N}$ define
\[
\psi_{\alpha,L}(v)=\left\{
\begin{array}{cc}
1-\frac{||v|-\alpha|}{L}   & {\rm if} \quad ||v|-\alpha|<L  \\
0 & {\rm otherwise.}
\end{array}
\right.\quad
\]
We denote
\[
j_{\alpha,L}(v)=\psi_{\alpha,L}(v)/\sqrt{w_{L}(Y)}, \quad w_{L}(v)=\sum_{\alpha}\psi^2_{\alpha,L}(v)
\]
so that $\{j_{\alpha,L}^2\}, \alpha\in \mathbb{N}$ forms a partition of unity, i.e.,
\[
\sum_{\alpha} j_{\alpha,L}^2(v)=1
\]
for every $v\in V(T)$. Note that this partition of unity is defined similarly to the one in \cite{LastSimonEss}, except that in our case it is
applied to ``annuli around $O$''.

It is easy to see that $\sqrt{w_L(v)}=c(L)$ for all $v$ with $|v|>L$. Moreover, $c(L)\sim \sqrt L$ in that there are constants, $c_1,c_2>0$
such that
\begin{equation} \nonumber
c_1 \sqrt{L} \leq c(L) \leq c_2 \sqrt{L}.
\end{equation}

As in \cite{LastSimonEss}, it now follows that
\[
|\left( [H,j_{\alpha,L}]\delta_{v_1},\delta_{v_2} \right) |\lesssim \left \{ \begin{array}{cc} (L\cdot c(L))^{-1} & \textrm{ if
dist}(v_1,v_2)=1, \textrm{ and } ||v_1|-\alpha|\le L+1 \\
0 & \textrm{ otherwise, } \end{array} \right.
\]
where the implied constant depends on $\sup_{u\sim v} a_{u,v}$ (recall the function $a$ from the definition of the Jacobi matrix $H$).

Consider the operator
$$
C=-2\sum_{\alpha} [H,j_{\alpha,L}]^2.
$$
It is a hopping operator with range at most $2$, i.e.,
$
\left( C v_1,v_2\right)=0
$ if ${\rm dist}(v_1,v_2)>2$ and the other matrix elements are bounded by
\begin{equation} \nonumber
2\cdot d \cdot(2L) (c(L))^{-2} L^{-2} \left(\sup_{u\sim v}a_{u,v} \right)^2.
\end{equation}
It follows that
$
\|C\|\lesssim L^{-2}.
$

We are now ready to finish the proof. Suppose $\lambda\in \sigma_{\rm ess}(H)$. Then by a simple adaptation of Weyl's criterion (as in the
proof of Theorem \ref{thm:Thm1} above), there is a sequence $\{f_n\}\in \ell^2(T)$ such that
\begin{equation}
\lim_{n\to\infty}\|(H-\lambda)f_n\|= 0, \quad \|f_n\|=1,
\end{equation}
and
\begin{equation}\label{sd1}
\sup_{|v|<R}|f_n(v)|=0
\end{equation}
for every $R>0$ and $n>n_R$.

Fix $\epsilon>0$ and choose $L_\epsilon$ so large that $\|C\|<\epsilon^2$ for every $L>L_\epsilon$. Then by Proposition \ref{prop:LastSimon}
\begin{equation} \label{eq:LastSimonCons}
\sum_{\alpha} \|(H-\lambda)j_{\alpha,L}f_n\|^2\lesssim \epsilon^2+\|(H-\lambda)f_n\|^2\lesssim \epsilon^2
\end{equation}
if $n>\widetilde n_\epsilon$.

In order to obtain an approximate eigenfunction of an $\mathcal{R}$-limit from what we have, we now need to exploit the structure of the tree.
For any $\alpha$ and $L$ let $\{O_\beta\}_{\beta=1}^{N_{\alpha,L}}$ be an enumeration of the vertices satisfying $|O_\beta|=\alpha-L$ ($N_{\alpha,L}=d\cdot d^{\alpha-L-1}$). For
each such $O_\beta$ let $T_\beta$ be the subtree of $T$ with root $O_\beta$ that is obtained by disconnecting the edge closest to $O_\beta$ on the path between $O$
and $O_\beta$. Notice that
\[
j_{\alpha,L}=\sum_{\beta=1}^{N_{\alpha,L}}j_{\alpha,L,\beta}
\]
where $j_{\alpha,L,\beta}=j_{\alpha,L}\cdot \chi_{T_\beta}$ and $\chi_{T_\beta}$ is the characteristic function of $T_\beta$. Notice further
that $j_{\alpha,L,\beta}(v)=0$ if $|v|=\alpha-L$ which implies that the functions $\{(H-\lambda)j_{\alpha,L,\beta}f_n\}$ have disjoint
supports and therefore
\[
 \|(H-\lambda)j_{\alpha,L}f_n\|^2= \sum_{\beta}\|(H-\lambda)j_{\alpha,L,\beta}f_n\|^2\,.
\]
Finally, we note that $\{j^2_{\alpha,L,\beta}\}$ forms a partition of unity, since  $j_{\alpha,L}^2=\sum_{\beta}j^2_{\alpha,L,\beta}$.

This yields that
\[
\sum_{\alpha,\beta} \|j_{\alpha,L,\beta}f_n\|^2=1
\]
and from \eqref{eq:LastSimonCons}
\[
\sum_{\alpha,\beta} \|(H-\lambda)j_{\alpha,L,\beta}f_n\|^2\lesssim  \epsilon^2\sum_{\alpha,\beta} \|j_{\alpha,L,\beta}f_n\|^2\,.
\]

This implies that for each sufficiently large $n$ there are $\alpha_n,\beta_n$ such that
\begin{equation}\label{sd2}
\|j_{\alpha_n,L,\beta_n}f_n\|> 0
\end{equation}
and
\[
 \|(H-\lambda)j_{\alpha_n,L,\beta_n}f_n\|\lesssim  \epsilon \|j_{\alpha_n,L,\beta_n}f_n\|
\]
Moreover $\lim_{n\to\infty}\alpha_n=\infty$ due to \eqref{sd1} and \eqref{sd2}.

Denote the restriction of $H$ to the support of $j_{\alpha_n,\beta_n}$ by $H_{\alpha_n,\beta_n}$. There is a vertex, $v_n$, such that
$B_{L/2}(v_n)$ is contained in the support of $j_{\alpha_n,\beta_n}$. In order to show that there is a subsequence of $\{v_n\}$ lying on a
path to infinity, let us identify each $v \in V(T)$ with a subinterval of $[0,1]$ as follows. Vertices with $|v|=1$ are each assigned to a
single interval $\left[\frac{j-1}{d},\frac{j}{d} \right]$ $(j=1,\ldots,d)$. Subdividing these intervals, each into $(d-1)$ subintervals, we
assign these intervals in turn to the vertices at distance $2$ from $O$ in such a way that if $v'$ lies in $T_v$ then its assigned interval is
contained in the interval assigned to $v$. Continuing in this way, we see that a path to infinity corresponds to a sequence of nested
intervals. It follows, as in the proof of the Bolzano-Weierstrass Theorem, that there is a subsequence of $\{v_n\}$ that lies on a path to
infinity. Thus, by compactness, there is an $\mathcal{R}$-limit, $H'$, of $H$ and a sequence, ${n_k}$, such that $H$ converges to $H'$ along the
sequence $\{v_{n_k}\}$ in the sense of Definition \ref{def:rlimdef}.

Therefore
\[
\|(H'-\lambda)h_n\|\lesssim \epsilon \|h_n\|
\]
for large enough $n$, where $h_n$ is a translation of $j_{\alpha_{k_n},L,\beta_{k_n}}f_{k_n}$. Since $\|h_n\|>0$ and $\epsilon$ is arbitrary,
we have
$$\lambda\in \overline{\bigcup_{L\textrm{ is an } \mathcal{R} \textrm{ limit of } H}\sigma(L)}$$ and the theorem follows.
\end{proof}


\section{Spherically Symmetric Operators on Trees} \label{sec:spherical}

Let $T$ be a $d$-regular tree, and as in the previous section, fix an arbitrary vertex and call it the root, $O$. A Jacobi matrix, $H$, on $T$
is spherically symmetric around $O$ if there exist functions
\begin{equation} \nonumber
A: \mathbb{N}\cup \{0\} \to (0,\infty), \quad B: \mathbb{N}\cup \{0\} \to \mathbb{R}
\end{equation}
so that
\begin{equation} \nonumber
a_{u,v}=A_{\min \left(|u|,|v| \right)}, \quad b(v)=B_{|v|}.
\end{equation}

The spherical symmetry implies that $H$ decomposes as a direct sum of Jacobi matrices on the half-line. Thus, its essential spectrum can be
described using this decomposition and the right limits of the constituent half-line operators. This section is devoted to a discussion of
this description and its relation to the one given by Theorem \ref{thm:Thm3}.

In order to describe the decomposition, we first need
\begin{definition}
\emph{The $k$-th tail} of a Jacobi matrix, $A$ on $\mathbb{N}$, is the Jacobi matrix, $A^{[k]}$ on $\mathbb{N}$, defined by
$\left(A^{[k]}\right)_{i,j}=(A)_{i+k,j+k}$. \emph{The sequence of tails of $A$} is the sequence of matrices $\{A^{[k]} \}_{k=0}^\infty$.
\end{definition}

Let $H$ be spherically symmetric on $T$. Then it follows (see \cite{AF, Breuer}) that $H$ is unitarily equivalent to a direct sum
\[\bigoplus_{n=1}^\infty\left(\oplus_{j=1}^{k_n} S_n\right),\]
where $\{S_n\}_n$ are Jacobi matrices on $\mathbb{N}$ with parameters $\{a_k^{(n)},b_k^{(n)}\}_{k=1}^\infty$ satisfying
\begin{equation*}
a_k^{(n)}=\begin{cases} \sqrt{d} A_1\ \ \ \ \ \ \ \ n=k=1 \\
\sqrt{d-1} A_{k+n-1}\ \ \ \textrm{otherwise},
\end{cases}
\end{equation*}
\begin{equation} \nonumber
b_k^{(n)}=B_{k+n-1},
\end{equation}
and $k_n$ is some explicit function of $n$ and the degree of the tree (\cite{AF, Breuer} have this decomposition for Schr\"odinger operators,
but using the ideas of \cite{Breuer} it is easy to extend the analysis to the Jacobi case). Notice that the direct sum includes $k_n$ copies
of $S_n$ for each $n\in\mathbb{N}$. Additionally, note that the matrix $S_n$ is the ($n-k$)'th tail of the matrix $S_k$ for any
$n>k\in\mathbb{N}$. In short,

\[
S_{1}=\left[\begin{array}{ccccc}
B_1 & \sqrt{d}A_1 & 0\\
\sqrt{d}A_1 & B_2 & \sqrt{d-1}A_2\\
0 & \sqrt{d-1}A_2 & B_3 & \sqrt{d-1}A_3\\
 &  & \sqrt{d-1}A_3 & B_4\\
 &  &  &  & \ddots
\end{array}\right]
\]
and all other $S_n$'s are tails of this Jacobi matrix.

\qquad{}The following general proposition deals with a situation of this type.\medskip{}
\begin{prop}\label{prop:PropA} Assume $J$ is a bounded Jacobi matrix on $\mathbb{N}$ with parameters $\{a_j\}_{j=1}^\infty$ and
$\{b_j\}_{j=1}^\infty$,
satisfying
\[
\sup_{j}\left(\left|a_{j}\right|+\left|b_{j}\right|\right)=M<\infty.
\]
Let $\left\{ J_{n}\right\} _{n=1}^{\infty}$ be a subsequence of the
sequence of tails of $J$, $\left\{J^{[k]}\right\}_{k=1}^\infty$, let $\left\{ i_{n}\right\}_{n=1}^\infty \subset\mathbb{N}$, and
let $K=\bigoplus_{n=1}^{\infty}\left(\oplus_{j=1}^{i_{n}}J_{n}\right)$.
Then the essential spectrum of $K$ satisfies:
\[
\sigma_{\text{ess}}\left(K\right)={\left(\bigcup_{r}\sigma\left(J^{\left(r\right)}\right)\right)\cup\left(\bigcup_{s}\sigma\left(J_{\left(s\right)}\right)\right)}
\]
where $\left\{ J^{\left(r\right)}\right\} $ is the set of right limits
of $J$, and $\left\{ J_{\left(s\right)}\right\} $ is the set of
limit points in the strong operator topology of the sequence $\left\{ J_{n}\right\} _{n=1}^{\infty}.$
\end{prop}
\vspace{0.3cm}

\noindent Before proving Proposition \ref{prop:PropA} we present another preliminary
proposition:
\noindent \medskip{}
\begin{prop}\label{prop:PropB}
The essential spectrum of $K$ satisfies:
\[
\sigma_{\text{ess}}(K)=\sigma_{\text{ess}}\left(J\right)\cup\Sigma
\]
where,
\begin{align*}
\Sigma_{\text{ }}=&\ \Sigma_0\big\backslash\sigma_{\text{ess}}\left(J\right), \\
\Sigma_0= &\left\{ \vphantom{\int_t} E\in\mathbb{R}\,|\,\exists\left(n_{k}\right)_{k=1}^{\infty}\subseteq\mathbb{N},n_{k+1}\geq
n_{k},\, n_k \rightarrow \infty, \, \left(g_{k}\right)\in\ell^{2}\left(\mathbb{N}\right),\,\lambda_{k}\in\mathbb{R}, \right.\\
&\, \left. \text{s.t.\ }   J_{n_{k}}g_{k}=\lambda_{k}g_{k},\text{ and }\lambda_{k}\underset{k\to\infty}{\longrightarrow}E\right\}
\end{align*}
$($i.e.\ $\Sigma$ is the set of limit points of eigenvalues of the $J_{n}$'s that are not in $\sigma_{\textrm{ess}})$.
\end{prop}
\begin{proof}[Proof of Proposition \ref{prop:PropB}]

First, note that $J_{n}$ is a finite rank perturbation of $J$ and thus
$\sigma_{\text{ess}}\left(J\right)=\sigma_{\text{ess}}\left(J_{n}\right)$
for every $n\in\mathbb{N}$. Thus $\sigma_{\text{ess}}\left(J\right)\subseteq\sigma_{\text{ess}}\left(K\right)$.
Additionally $\Sigma \subseteq \sigma(K)$ and its elements are either eigenvalues of infinite multiplicity or non-isolated points of $\sigma(K)$, so $\Sigma\subseteq\sigma_{\text{ess}}\left(K\right)$.
Thus,
\[
\sigma_{\text{ess}}(K)\supseteq\sigma_{\text{ess}}\left(J\right)\cup\Sigma.
\]
For the reverse inclusion, denote $\sigma_{n}=\sigma\left(J_{n}\right)\backslash\sigma_{\text{ess}}\left(J_{n}\right)$, so,
\begin{align*}
\sigma\left(J_{n}\right)=\sigma_{\text{ess}}\left(J_{n}\right)\cup\sigma_{n}=\sigma_{\text{ess}}\left(J\right)\cup\sigma_{n}.
\end{align*}
Then,
\begin{align*}
\sigma\left(K\right)&=\overline{\bigcup_{n}\sigma\left(J_{n}\right)}=\overline{\bigcup_{n}\left(\sigma_{\text{ess}}\left(J\right)\cup\sigma_{n}\right)}=\\
&=\overline{\sigma_{\text{ess}}\left(J\right)\cup\left(\bigcup_{n}\sigma_{n}\right)}.
\end{align*}
The essential spectrum is closed and thus we can write
\[
\sigma\left(K\right)=\sigma_{\text{ess}}\left(J\right)\cup\overline{\bigcup_{n}\sigma_{n}},
\]
and we claim that this is exactly:
\[
=\left(\bigcup_{n}\sigma\left(J_{n}\right)\right)\cup\Sigma.
\]
Indeed, $\sigma_{\text{ess}}\left(J\right)\subseteq\sigma\left(J_{n}\right)$
for every $n$, and if $\lambda\in\left(\overline{\bigcup_{n}\sigma_{n}}\right)\Big\backslash\sigma_{\text{ess}}\left(J\right)$
then it is either an isolated eigenvalue of some $J_{n}$ (so
$\lambda\in\sigma\left(J_{n}\right)$), or an accumulation point
of eigenvalues of $J_{n}$'s, in which case $\lambda\in\Sigma$.
The opposite inclusion is immediate. We now have
\[
\sigma\left(K\right)\backslash\left(\sigma_{\text{ess}}\left(J\right)\cup\Sigma\right)=\left[\left(\bigcup_{n}\sigma\left(J_{n}\right)\right)\cup\Sigma\right]\bigg\backslash\left(\sigma_{\text{ess}}\left(J\right)\cup\Sigma\right)=
\]
\[
=\left[\bigcup_{n}\left(\sigma\left(J_{n}\right)\backslash\sigma_{\text{ess}}\left(J\right)\right)\right]\bigg\backslash\Sigma=\left(
\bigcup_n\sigma_n\right)\backslash \Sigma.
\]
Each term
$\sigma_{n}=\sigma\left(J_{n}\right)\backslash\sigma_{\text{ess}}\left(J\right)=\sigma\left(J_{n}\right)\backslash\sigma_{\text{ess}}\left(J_{n}\right)=\sigma_{\text{disc}}\left(J_{n}\right)$
contains only isolated eigenvalues of finite multiplicity. Every accumulation
point of such points is contained in $\Sigma$. Hence, every point
in $\left(\cup_{n}\sigma_n\right)\backslash\Sigma$
is an isolated eigenvalue of finite multiplicity \textbf{of finitely
many} \textbf{$J_{n}$'s}, and thus it is also an isolated eigenvalue of finite multiplicity of $K$. We conclude that
$\sigma\left(K\right)\backslash\left(\sigma_{\text{ess}}\left(J\right)\cup\Sigma\right)\subseteq\sigma_{\text{disc}}\left(K\right)$,
and thus $\sigma_{\text{ess}}\left(K\right)\subseteq\sigma_{\text{ess}}\left(J\right)\cup\Sigma$.
\end{proof}

\begin{proof}[Proof of Proposition \ref{prop:PropA}]
We shall show that
\begin{equation} \nonumber
{\left(\bigcup_{r}\sigma\left(J^{\left(r\right)}\right)\right)\cup\left(\bigcup_{s}\sigma\left(J_{\left(s\right)}\right)\right)}
=\sigma_{\text{ess}}\left(J\right)\cup\Sigma
\end{equation}
which will imply the result by Proposition \ref{prop:PropB}.

The inclusion
\begin{equation} \nonumber
{\left(\bigcup_{r}\sigma\left(J^{\left(r\right)}\right)\right)\cup\left(\bigcup_{s}\sigma\left(J_{\left(s\right)}\right)\right)}
\subseteq \sigma_{\text{ess}}\left(J\right)\cup\Sigma
\end{equation}
follows from the fact that by Theorem \ref{thm:LastSimon}
\begin{equation} \nonumber
\left(\bigcup_{r}\sigma\left(J^{\left(r\right)}\right)\right) \subseteq \sigma_{\text{ess}}\left(J\right)
\end{equation}
and
\begin{equation} \nonumber
\bigcup_{s}\sigma\left(J_{\left(s\right)}\right)\subseteq \sigma_{\text{ess}}\left(J\right)\cup\Sigma
\end{equation}
since every $\lambda \in \sigma\left(J_{\left(s \right)} \right)$ is a limit point of a sequence $\{\lambda_n \}_n$ with $\lambda_n \in \sigma
\left(J_n \right)$.

For the reverse inclusion, again by Theorem \ref{thm:LastSimon} we have that
$\sigma_{\text{ess}}\left(J\right)\subseteq{\cup_{r}\sigma\left(J^{\left(r\right)}\right)}$.
Thus (using Proposition \ref{prop:PropB}) it is sufficient to prove that $\Sigma\subseteq{\cup_{s}\sigma\left(J_{\left(s\right)}\right)}$.
Let $E\in\Sigma$. Assume $\left\{ E_{k}\right\} $ is a sequence of eigenvalues  of $\left\{ J_{n_{k}}\right\} $, with $n_{k+1}\geq n_{k}$,
s.t.\ $E_{k}\to E$, and let $\psi_{k}$ be the corresponding eigenfunctions, satisfying $J_{n_{k}}\psi_{k}=E_{k}\psi_{k}$, $\left\Vert
\psi_{k}\right\Vert =1$. Since $n_{k}\to\infty$ we may assume, by restricting to a subsequence if
necessary, that $J_{n_{k}}$ converges strongly to some $J_{\left(s\right)}$, i.e.\ for any $\psi\in\ell^2\left(\mathbb{N}\right)$,
$\left\Vert J_{\left(s\right)}\psi-J_{n_{k}}\psi\right\Vert \to0$.

Denote by $\mu_{k}$ the spectral measure of $J_{n_{k}}$ and $\delta_{1}=(1,0,0,0,\ldots)$.
Then,
\[
\mu_{k}\underset{w}{\longrightarrow}\mu_{s}
\]
where $\mu_{s}$ is the spectral measure of $J_{\left(s\right)}$
and $\delta_{1}$.

\noindent Assume first that $\lim_{k}\mu_{k}\left(\left\{ E_{k}\right\} \right)=0$. We shall show that in this case $E \in
\sigma_{\textrm{ess}}(J)$, in contradiction with $E \in \Sigma$. Note (letting $\chi_{E}\left( J_{n_k}\right)$ be the spectral projection of $J_{n_k}$ onto $\left\{E \right \}$)
\[
\left|\psi_{k}(1)\right|^{2}=\left(\delta_1,\psi_k \right)\left(\psi_k,\delta_1 \right)=\left(\delta_1,\chi_{E_k}\left(J_{n_k} \right) \delta_1 \right)=\mu_{k}\left(\left\{ E_{k}\right\} \right)\underset{k\to\infty}{\longrightarrow}0.
\]
Define $\left\{ \widetilde{\psi}_{k}\right\} _{n=1}^{\infty}$ by
\begin{equation} \nonumber
\widetilde{\psi}_{k}(j)=\begin{cases}
0\,\,\, & j<n_{k}\\
\psi_{k}(j-n_{k}+1)\,\: & j\geq n_{k}
\end{cases}.
\end{equation}
Then $\widetilde{\psi}_{k}$ satisfies $\left\Vert \widetilde{\psi}_{k}\right\Vert =1$
and,
\[
\left(J\widetilde{\psi}_{k}\right)\left(j\right)=\begin{cases}
\left(J_{n_{k}}\psi_{k}\right)\left(j-n_{k}+1\right)\,\,\, & j\geq n_{k}\\
a_{n_{k}}\psi_{k}\left(1\right) & j=n_{k}-1\\
0 & j<n_{k}-1
\end{cases}.
\]
Thus,
\[
\left\Vert J\widetilde{\psi}_{k}-E\widetilde{\psi}_{k}\right\Vert \leq\left\Vert
J\widetilde{\psi}_{k}-\widetilde{J_{n_{k}}\psi_{k}}\right\Vert +\left\Vert \widetilde{J_{n_{k}}\psi_{k}}-E\widetilde{\psi}_{k}\right\Vert \leq
\]
\[
\leq\left|a_{n_{k}}\psi_{k}\left(1\right)\right|+\left|E_{k}-E\right|\to0.
\]
In addition, it is clear that $\widetilde{\psi}_{k}\underset{\text{w}}{\longrightarrow}0$, which implies that $E \in
\sigma_{\textrm{ess}}(J)$. Thus, we can conclude that
\[\overline{\lim_{k}}\mu_{k}\left(\left\{ E_{k}\right\} \right)>0,
\]
so by taking a subsequence of $\left\{E_{k}\right\}$ we can assume that
\[\nu=\lim_{k}\mu_{k}\left(\left\{E_{k}\right\}\right)>0\]
exists.
Let $\varepsilon>0$ and $f\in\mathcal{C}\left(\mathbb{R}\right)$
s.t.\ $\text{supp}\left(f\right)\subseteq\left(E-\varepsilon,E+\varepsilon\right)$,
$f\geq0$ and $f\left(E\right)>0$. Then there exists some $N\in\mathbb{N}$ s.t.\ for every $k>N$,
\[
\left|E_{k}-E\right|<\varepsilon,\ \mu_{k}\left(\left\{ E_{k}\right\} \right)>\sfrac{\nu}{2}\ \text{and }
f\left(E_{k}\right)>\sfrac{f\left(E\right)}{2}.
\]
Now, for every $k>N$ we have that
\[
\int f\,d\mu_{k}\geq f\left(E_{k}\right)\mu_{k}\left(\left\{ E_{k}\right\} \right)\geq\sfrac{f(E)\cdot \nu}{4}>0.
\]
Hence by the weak convergence of the measures $\mu_{k}\underset{w}{\longrightarrow}\mu_{s}$
we conclude that $\int f\,d\mu_{s}\geq\sfrac{f(E)\cdot \nu}{4}>0$
for every such $f$, and thus (since the spectrum is a closed set)
$E\in\overline{\text{supp}\left(\mu_{s}\right)}\subseteq\sigma\left(J_{\left(s\right)}\right)$.
\end{proof}

It follows that in the spherically symmetric case, the one-dimensional decomposition provides one with an additional description of the
essential spectrum. However, in contrast to the one-dimensional case, the essential spectrum is given by spectra of one sided matrices in
addition to the whole-line ones (obtained as the right limits). Thus, for example in the case of periodic $\{A_j,B_j\}$ we would get as the
essential spectrum not only the spectrum of the corresponding periodic whole-line Jacobi matrix, but also the possible eigenvalues from the
half-line matrix.

Comparing Proposition \ref{prop:PropA} with Theorem \ref{thm:Thm3}, we conclude that the spectrum of an $\mathcal{R}$-limit of $H$ in the
spherically symmetric case is described using the spectra of the one-dimensional Jacobi matrices derived from the decomposition of $H$, as in
Proposition \ref{prop:PropA}. In fact it is possible to make an identification:

\begin{theorem} \label{thm:Thm4}
Let $H$ be a spherically symmetric Jacobi matrix on a regular tree and let
\[
K=H=\bigoplus_{n=1}^{\infty}\left(\oplus_{j=1}^{k_{n}}S_{n}\right)
\]
be its decomposition to half-line Jacobi matrices.

Let further $\left \{\left(J^{\left(r\right)}\right) \right \}$, $\left\{ \left(J_{\left(s\right)}\right)\right\}$ be the corresponding
limiting matrices defined in Proposition \ref{prop:PropA}. Then for any $J^{(r)}$ there exists an $\mathcal{R}$-limit
$\left\{L,T',v_0\right\}$ of $H$ so that $\sigma\left(J^{\left(r\right)}\right)\subseteq\sigma\left(L\right)$, and for any $J_{(s)}$ there
exists an $\mathcal{R}$-limit $\left\{\widetilde{L},T',\widetilde{v_0}\right\}$ of $H$ so that
$\sigma\left(J_{\left(s\right)}\right)\subseteq\sigma\left(\widetilde{L}\right).$
\end{theorem}

Note that even if $H$ is spherically symmetric, its $\mathcal{R}$-limits need not be. Thus, we find this correspondence between
$\mathcal{R}$-limits and one-dimensional Jacobi matrices quite remarkable. The \emph{constructive} proof relies on a `spherical decomposition
with respect to a point at infinity' in some sense. It is given in Appendix B.

Finally, note that Theorem \ref{thm:Thm4} implies that for spherically symmetric $H$
\begin{equation} \label{eq:EssNoClosure}
\sigma_{\text{ess}}\left(H\right)=\bigcup_{L\text{ is an \ensuremath{\mathcal{R}}-limit of \ensuremath{H}}}\sigma\left(L\right)
\end{equation}
which shows that the conclusion of Theorem \ref{thm:Thm3} is true in this case without the closure.




\section{Appendix A: calculating the spectrum of $A_{\tilde T}$}
\begin{proof}[Proof of Lemma \ref{lem:Lemma1}]
The tree $\widetilde{T}=\widetilde{T}_d$ is composed of a $d$-regular tree $T=T_{d}$
connected at a point $0\in T$ to the first point $1\in\mathbb{N}$
in a line. Thus $A_{\widetilde{T}}$ is a finite rank perturbation of $A_T\oplus A_\mathbb{N}$, and so
\[\sigma_\text{ess}\left(A_{\widetilde{T}}\right)=\sigma_\text{ess}\left(A_T\right)\cup\sigma_\text{ess}\left(A_\mathbb{N}\right)=\sigma_\text{ess}\left(A_T\right)=[-2\sqrt{d-1},2\sqrt{d-1}].
\]
Therefore $d\notin\sigma_\text{ess}\left(A_{\widetilde{T}}\right)$. We want to exclude the possibility that $d \in
\sigma\left(A_{\widetilde{T}} \right) \setminus \sigma_{\textrm{ess}}\left(A_{\widetilde{T}} \right)$.

Using Dirac's bra-ket notation, define $A_{0}=A_{\widetilde{T}}-\left|\delta_{0}\right\rangle \left\langle \delta_{1}\right|-\left|\delta_{1}\right\rangle \left\langle
\delta_{0}\right|$
and $R\left(z\right)=\left(A_{\widetilde{T}}-z\right)^{-1}$, $R_{0}\left(z\right)=\left(A_{0}-z\right)^{-1}$.
Recall the resolvent identity (we omit the dependence on $z$),
\begin{equation} \label{eq:resid}
R_{0}-R=R_{0}\left(A_{\widetilde{T}}-A_{0}\right)R=R_{0}\left(\left|\delta_{0}\right\rangle \left\langle
\delta_{1}\right|+\left|\delta_{1}\right\rangle \left\langle \delta_{0}\right|\right)R.
\end{equation}
Multiplying by $\delta_{0}$ on both sides we have
\[
m_{T}(z)-m(z)=m_{T}(z)\left\langle \delta_{1}\right|R(z)\left|\delta_{0}\right\rangle,
\]
where
\[
m_{T}\left(z\right)=\left\langle \delta_{0}\right|R_{0}\left(z\right)\left|\delta_{0}\right\rangle
\]
\[
m_{\mathbb{N}}\left(z\right)=\left\langle \delta_{1}\right|R_{0}\left(z\right)\left|\delta_{1}\right\rangle
\]
\[
m\left(z\right)=\left\langle \delta_{0}\right|R\left(z\right)\left|\delta_{0}\right\rangle .
\]
Additionally, by multiplying the identity \eqref{eq:resid} by $\delta_{1}$ on the left and by $\delta_{0}$ on the right we have
\[
0-\left\langle \delta_{1}\right|R(z)\left|\delta_{0}\right\rangle =m_{\mathbb{N}}(z)m(z).
\]
Combining we get,
\[
m_{T}(z)-m(z)=m_{T}(z)\left(-m_{\mathbb{N}}(z)m(z)\right),
\]
which implies
\begin{equation} \label{eq:appendix1}
m(z)=\frac{m_{T}(z)}{1-m_{T}(z)m_{\mathbb{N}}(z)}.
\end{equation}

Now, if $\lambda\in\sigma\left(A_{\widetilde{T}}\right)\Big\backslash\sigma_\text{ess}\left(A_{\tilde{T}}\right)$, then
$\lim_{\varepsilon\to0}\text{Im}\left(m\left(\lambda+i\varepsilon\right)\right)\neq0$.
We will consider this expression for $\lambda=d$. It follows from \eqref{eq:appendix1} that
\[
\text{Im}\left(m\right)=\frac{\text{Im}\left(m_{T}\left(1+\overline{m_{T}m_{\mathbb{N}}}\right)\right)}{\left|1-m_{T}m_{\mathbb{N}}\right|^{2}}=\frac{\text{Im}\left(m_{T}\right)-\left|m_{T}\right|^{2}\text{Im}\left(m_{\mathbb{N}}\right)}{\left|1-m_{T}m_{\mathbb{N}}\right|^{2}}.
\]

It is known (see, e.g., \cite{SimonSz}) that
\[
m_{\mathbb{N}}\left(z\right)=\frac{-z+\sqrt{z^{2}-4}}{2}
\]
\[
m_{T}\left(z\right)=\frac{-2\left(d-1\right)}{\left(d-2\right)z+d\sqrt{z^{2}-4\left(d-1\right)}}.
\]
Thus
\[
\lim_{\varepsilon\to0}\text{Im}\left(m_{T}\left(d+i\varepsilon\right)\right)=\lim_{\varepsilon\to0}\text{Im}\left(m_{\mathbb{N}}\left(d+i\varepsilon\right)\right)=0.
\]
Additionally the denominator of $\text{Im}\left(m\right)$ satisfies,
\begin{align*}
1-m_{T}\left(d+i0\right)m_{\mathbb{N}}\left(d+i0\right)=&1-\left(d-1\right)\frac{d-\sqrt{d^{2}-4}}{\left(d-2\right)d+d\sqrt{d^{2}-4\left(d-1\right)}}=\\
&1-\frac{\left(d-1\right)\left(d-\sqrt{d^{2}-4}\right)}{2d\left(d-2\right)} \neq 0
\end{align*}
since the number $\sqrt{d^{2}-4}$ is
irrational for every $2<d\in\mathbb{N}$ ($d^2-4$ is not a perfect square). Thus
\begin{equation} \nonumber
\lim_{\varepsilon\to0}\left|1-m_{T}\left(d+i\varepsilon\right)m_{\mathbb{N}}\left(d+i\varepsilon\right)\right|^{2}>0,
\end{equation}
and we get
\[
\lim_{\varepsilon\to0}\text{Im}\left(m\left(d+i\varepsilon\right)\right)=0.
\]
This implies that $d\notin\sigma\left(A_{\widetilde{T}}\right)\Big\backslash\sigma_\text{ess}\left(A_{\widetilde{T}}\right)$,
and we can conclude that $d\notin\sigma\left(A_{\widetilde{T}}\right)$.
\end{proof}

\begin{remark*} As mentioned in \ref{rem:Httspec}, actually $\sigma\left(A_{\widetilde{T}}\right)=\sigma\left(A_T\right)$.
The inclusion $\sigma\left(A_{T}\right)\subseteq\sigma\left(A_{\widetilde{T}}\right)$ is clear.
Additionally, it is not hard, but is a bit cumbersome to see that for any $\lambda\notin\sigma\left(A_T\right)$ the expression
$1-m_{T}\left(\lambda\right)m_{\mathbb{N}}\left(\lambda\right)$ is nonzero, and thus in this case also
$\lambda\notin\sigma\left(A_{\widetilde{T}}\right)$.
\end{remark*}

\section{Appendix B: Proof of Theorem 5}

We begin by discussing some properties of $\mathcal{R}$-limits of regular trees.
Let $H$ be a Jacobi matrix on a $d$-regular tree $T$ (with $d>2$) with root vertex $O\in T$.
Let $\left\{H',T',v_0'\right\}$ be an $\mathcal{R}$-limit of $H$ along a path to infinity ${\left\{v_j\right\}_{j=0}^\infty\subset
V\left(T\right)}$. Since a regular tree is homogeneous, any $\mathcal{R}$-limit of $H$ is defined on the same regular tree. Thus $T'$ is
another copy of the $d$-regular tree, which, for convenience, we take to be distinct from $T$.

\begin{definition}
Given an isometry between trees $f:T\to T'$, denote by $I_{f}:\ell^{2}\left(T\right)\to\ell^{2}\left(T'\right)$
the isometry operator: $\left(I_{f}\psi\right)\left(v\right)=\psi\left(f\left(v\right)\right)$.
\end{definition}
\begin{definition}
For any vertex $u\in T$, $R>0$, denote by $P_{u,R}$ the projection operator onto $\ell^{2}\left(B_{R}\left(u\right)\right)$.
Further, for any operator $X$ on $\ell^{2}\left(T\right)$, let $X_{u,R}=P_{u,R}XP_{u,R}$.
\end{definition}

\begin{prop} \label{prop:RlimOnRegT}
Let $H$ be a Jacobi matrix on a $d$-regular tree $T$, and assume $\left\{H',T',v_0'\right\}$ is an $\mathcal{R}$-limit of $H$ along a path to
infinity $\left\{v_j\right\}_{j=0}^\infty$. Then there exists a subsequence of vertices
$\left\{u_j\right\}_{j=1}^\infty\subseteq\left\{v_j\right\}_{j=0}^\infty$ and a sequence of tree isometries
$\left\{f_{j}:T\to T'\right\}_{j=1}^\infty$, with $f_{j}\left(u_{j}\right)=v_{0}'$ (see Figure \ref{figure2}) satisfying, for every $R>0$,
\begin{equation} \label{eq:rlimonregT}
\left\Vert I_{f_{j}}H_{u_{j},R}I_{f_{j}}^{-1}-H'_{v_{0}',R}\right\Vert \underset{j\to\infty}{\longrightarrow}0.
\end{equation}

Moreover, if $H'$ is a Jacobi matrix on $T'$ and there exist sequences $\left\{u_j\right\}_{j=1}^\infty$, and
$\left\{f_j\right\}_{j=1}^\infty$ as above, s.t.\ \eqref{eq:rlimonregT} is satisfied for any $R>0$, then $\left\{H',T',v_0'\right\}$ is an
$\mathcal{R}$-limit of $H$ along the path $\left\{v_j\right\}_{j=0}^\infty$.
\end{prop}

\begin{figure}[ht!]
\centering
\includegraphics[width=130mm]{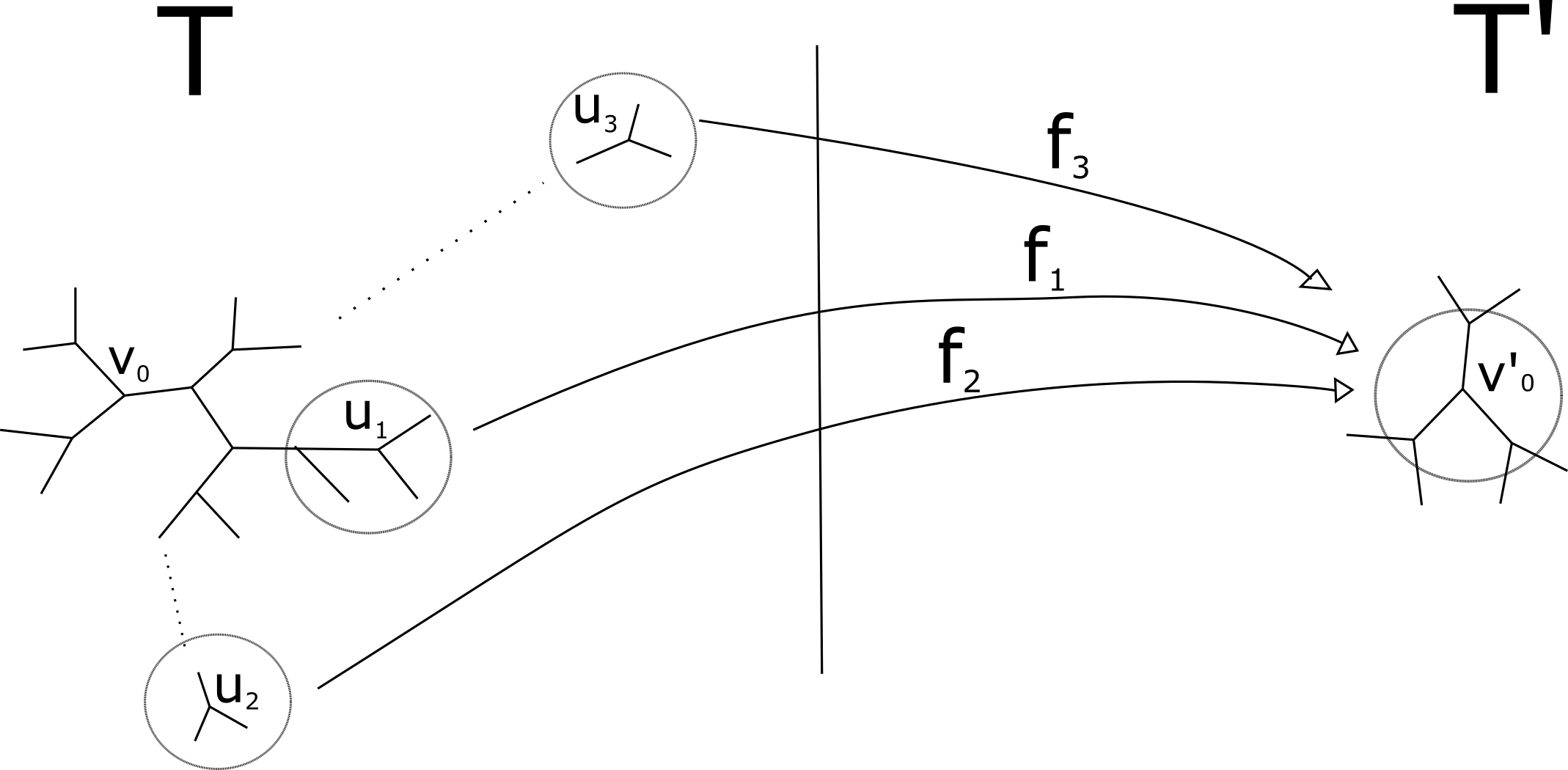}\\
\caption{The isometries $f_j$ between $T$ and $T'$. \label{figure2}}
\end{figure}

\begin{proof}
Assume the $\mathcal{R}$-limit $H'$ is obtained along the subsequence $\left\{v_{n_j}\right\}_{j=1}^\infty$ (as in Definition
\ref{def:rlimdef}), and define $u_j=v_{n_j}$. Note that for any $j\in\mathbb{N}$ the coherent isomorphisms sequence
$\left\{\mathcal{I}^{(j)}_k \right\}_{k=1}^\infty$ can be extended to an isomorphism
$\mathcal{I}_j:\ell^2\left(T\right)\to\ell^2\left(\mathbb{N}\right)$, that agrees on balls around $u_j$ with $\mathcal{I}^{(j)}_k$. Similarly,
the sequence $\mathcal{I'}_k$ can be extended to an isomorphism $\mathcal{I'}:\ell^2\left(T'\right)\to\ell^2\left(\mathbb{N}\right)$. Now we
can define $I_{f_j}$ (and $f_j$) by $I_{f_j}={\mathcal{I'}^{-1}\mathcal{I}_{j}}$. The convergence \eqref{eq:rlimonregT} then follows directly
from \eqref{eq:rlimdef}.

In the other direction, assume $\left\{u_j\right\}$ is a sequence of vertices and $\left\{f_j\right\}_{j=1}^\infty$ is a sequence of tree
isometries as above.
By compactness there is a path to infinity, $\left\{v_j\right\}_{j=1}^\infty$, which contains a subsequence
$\left\{u_j'\right\}_{j=1}^\infty\subseteq\left\{u_j\right\}_{j=1}^\infty$, i.e.\ $u_j'=v_{n_j}$, for a corresponding sequence
$\left\{n_j\right\}\subseteq\mathbb{N}$. Let $\left\{\mathcal{I'}_k\right\}_{k=1}^\infty$ be any sequence of coherent isomorphisms of $T'$
around $v_0'$. We can now define for any $j\in\mathbb{N}$ a coherent sequence of isomorphisms
$\left\{\mathcal{I}^{(j)}_{k}\right\}_{k=1}^\infty$ of $T$ around $u_j'$ by
\begin{equation} \nonumber
\mathcal{I}^{(j)}_{k}=\mathcal{I'}_k \mathcal{I}_{f_j}|_{B_k\left(u_j'\right)}.
\end{equation}
The convergence \eqref{eq:rlimdef} follows directly from  \eqref{eq:rlimonregT}.
\end{proof}

As we show next, it is possible to choose the isometries $\left\{f_j\right\}$ s.t.\ the path from $u_j$ to $v_0$ is always mapped to the same
sequence of vertices in $T'$ from $v_0'$ to infinity.

\begin{definition}
Denote by $N\left(u\right)$ the set of \textbf{neighbors} of the vertex $u$. Additionally, assuming $u\neq v_0$, denote by
$A\left(u\right)=A^1(u)=A_{v_0}\left(u\right)$ the vertex ${w\in N(u)}$ on the (shortest) path from $v_0$ to $u$. For $n\in \mathbb{N}$, let
$A^n(v)=A\left(A^{n-1}\left(v \right)\right)$.
\end{definition}

\begin{prop} \label{prop:ancestorsSeq}
Let $\left\{T,v_0\right\}$ and $\left\{T',v_0'\right\}$ be two copies of the $d$-regular tree, assume $\left\{v_j\right\}_{j=0}^\infty$ is a
path to infinity in $T$, $\left\{w_j\right\}_{j=1}^\infty\subseteq\left\{v_j\right\}_{j=0}^\infty$ is a subsequence, and  $\left\{g_{j}:T\to
T'\right\}_{j=1}^\infty$ is a sequence of tree isometries, with $g_j\left(w_j\right)=v_0'$. Then there exist a subsequence of vertices
$\left\{u_j\right\}_{j=1}^\infty\subseteq\left\{w_j\right\}_{j=1}^\infty$, a corresponding subsequence of tree isometries
$\left\{f_{j}\right\}_{j=1}^\infty\subseteq\left\{g_{j}\right\}_{j=1}^\infty$, and a path to infinity
${\left\{v_k'\right\}_{k=0}^\infty\subset V\left(T'\right)}$, s.t.\
\begin{equation} \label{eq:ancestorSeqCond}
\forall n\in\mathbb{N},\ f_{j}\left(A^{n}\left(u_{j}\right)\right)=v_{n}',
\end{equation}
for any $j\in\mathbb{N}$ s.t.\ $\left|u_{j}\right|>n$.
\end{prop}

\begin{proof}
By compactness, the sequence $\left\{ f_{j}\left(A\left(w_{j}\right)\right)\right\} _{j=1}^{\infty}\subset T'$ contains a vertex ${v\in
N(v_{0}')}$ an infinite number of times, denote it by $v_{1}'=v$
and restrict to this sebsequence. Continue further inductively
to define the path to infinity $\{v_{k}'\}_{k=0}^{\infty}$. Finally take the diagonal over the resulting subsequences of vertices
$\subseteq\left\{w_j\right\}_{j=1}^\infty$ and tree isometries $\subseteq\left\{g_j\right\}_{j=1}^\infty$ to define the subsequence
$\left\{u_{j}\right\}_{j=1}^\infty$ and the subsequence $\left\{f_{j}\right\}_{j=1}^\infty$.
\end{proof}

We refer to the sequence $\{v_{k}'\}_{j=0}^{\infty}$ from Proposition \ref{prop:ancestorsSeq} as an \emph{ancestors sequence}.

\begin{corollary}
Let $H$ be a Jacobi matrix on a $d$-regular tree $T$, and assume $\left\{H',T',v_0'\right\}$ is an $\mathcal{R}$-limit of $H$  along a path to
infinity ${\left\{v_j\right\}_{j=0}^\infty}$.
Then there exists a subsequence of vertices $\left\{u_j\right\}_{j=1}^\infty\subseteq\left\{v_j\right\}_{j=0}^\infty$, a sequence of tree
isometries $\left\{f_{j}\right\}_{j=1}^\infty$ as in Proposition \ref{prop:RlimOnRegT}, and an ancestors sequence
${\left\{v_k'\right\}_{k=0}^\infty\subset V\left(T'\right)}$, s.t.\ \eqref{eq:rlimonregT} and \eqref{eq:ancestorSeqCond} are satisfied.
\end{corollary}

As the term `ancestors' suggests, an ancestors sequence defines a natural direction on $T'$ that we want to think of as a direction `toward the past' (or toward the tree $T$). Considering $v_{k+1}'$ as an ancestor of $v_k'$, we obtain a partial order relation on the tree comparing ancestors and their `descendants'. Explicitly, given an ancestors sequence $\left\{ v_{k}'\right\} _{k=0}^{\infty}$ in $T'$, and $k \in \mathbb{N} \cup \{0\}$, the subtree $\Gamma_{v_k'}$ is the connected component containing $v_k'$ in the graph obtained from $T'$ by removing the edge $(v_{k+1}', v_k')$. For any $v \in \Gamma_{v_k'}$ such that $v \neq v_k'$ we write $v_k'>_{D} v$ (thus $\Gamma_{v_k'}$ is the tree of `descendants' of $v_k'$, together with  the ancestor $v_k'$). Moreover, if $v \in \Gamma_{v_k'}$ for some $k$, let $\Gamma_v$ be the connected component containing $v$ in the graph obtained from $\Gamma_{v_k'}$ by removing the edge on the unique shortest path between $v$ and $v_k'$. $\Gamma_v$ is the tree of descendants of $v$ and we write $v >_{D} u$ for any $u \in \Gamma_v$ with $u\neq v$. The tree structure implies that this is well defined (i.e.\ $\Gamma_v$ does not depend on $k$ and the definitions coincide for $v = v_k'$) so that we get a partial order relation on $T'$.

With this notation in place, we are finally ready for the

\begin{proof}[Proof of Theorem \ref{thm:Thm4}]

Let $J^{(r)}$ be a right limit of $J=S_{1}$ along a sequence
$\left\{ l_{i}\right\} _{i=1}^{\infty}\subseteq\mathbb{N}$, i.e.
\begin{equation} \label{eq:Jrightlim}
\left\Vert J_{l_{i},R}-J_{0,R}^{\left(r\right)}\right\Vert \underset{i\to\infty}{\longrightarrow}0
\end{equation}
for any $R>0$. We claim that we can find a corresponding $\mathcal{R}$-limit
$\left\{L,T',v_0\right\}$ of $H$, s.t.\ $\sigma\left(J^{\left(r\right)}\right)\subseteq\sigma\left(L\right)$.
Indeed, for $i\in\mathbb{N}$ take some $u_{i}\in T$ s.t.\ $\text{\text{dist}}\left(u_{i},v_{0}\right)=l_{i}$,
and take any isomorphism of trees $f_{i}:T\to T'$ s.t.\ $f_{i}\left(u_{i}\right)=v_{0}'$.
Moreover, by Proposition \ref{prop:ancestorsSeq} and restricting to a subsequence if necessary, we can assume the existence of an ancestors
sequence $\left\{ v_{j}'\right\} _{j=0}^{\infty}$ s.t.\ \eqref{eq:ancestorSeqCond} is satisfied.
For any $j\in\mathbb{N}$ define a Jacobi matrix $L^{\left(j\right)}$ on $\ell^2\left(\Gamma_{v_j'}\right)$, with diagonal terms:
\[
\left(L^{\left(j\right)}\right)_{x,x}=\left(J^{\left(r\right)}\right)_{\left|x\right|-j,\left|x\right|-j},
\]
and off diagonal terms
\[
\left(L^{\left(j\right)}\right)_{x,y}=\frac{1}{\sqrt{d-1}}\left(J^{\left(r\right)}\right)_{\left|x\right|-j,\left|y\right|-j},
\]
where for $x\in\Gamma_{v_j'}$, $\left|x\right|=\text{dist}\left(x,v_j'\right)$. Note that $L^{\left(j\right)}$ is spherically symmetric around
$v_j'$, and that the sequence $\left\{L^{\left(j\right)}\right\}_{j=1}^\infty$ satisfies
$L^{(j)}|_{\Gamma_{v'_{j}}}=L^{(k)}|_{\Gamma_{v'_{j}}}$ for $k\geq j$ (see Figure \ref{figure3}).
\begin{figure}[ht!]
\centering
\includegraphics[width=130mm]{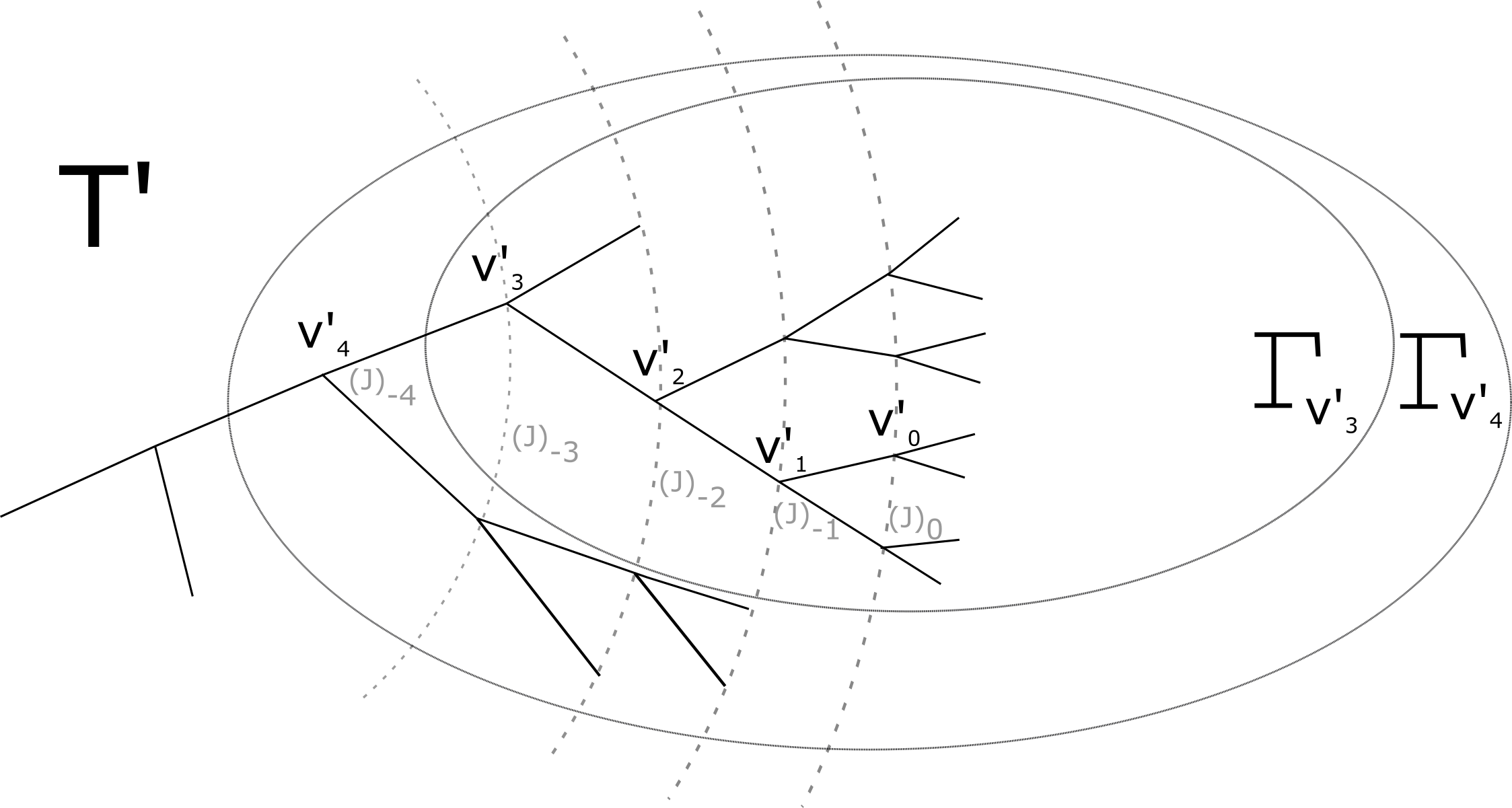}\\
\caption{The subtrees $\Gamma_{v'_j}$, and the diagonal terms of $L$, $(J)_k\equiv\left(J^{(r)}\right)_{k,k}$. \label{figure3}}
\end{figure}
Thus, we may define $L$ on $\ell^2\left(T'\right)$ by $L|_{\Gamma_{v'_{j}}}=L^{\left(j\right)}$. This defines an operator on
$\ell^2\left(T'\right)$ since $\cup_{j=1}^\infty \Gamma_{v_j'}=T'$. The sequence $\left\{L^{\left(j\right)}\right\}_{j=1}^\infty$ converges
strongly to $L$: indeed, for any $\varepsilon>0$ and $g\in\ell^{2}\left(T'\right)$ we can find $R>0$ s.t.\ $\left\Vert g|_{T'\backslash
B_{R}\left(v_{0}'\right)}\right\Vert <\varepsilon$,
and thus for any $j>R+1$,
\[
\left\Vert \left(L-L^{\left(j\right)}\right)g\right\Vert <\left\Vert L\right\Vert \left\Vert g|_{T'\backslash
B_{R}\left(v_{0}'\right)}\right\Vert <\left\Vert L\right\Vert \varepsilon.
\]

Next notice that by the spherical decomposition (and the symmetry) there exists for any $i\in\mathbb{N}$ a map
$n_i(x):B_R\left(u_i\right)\to\mathbb{N}\cap[1,2R+1]$, s.t.\ for $x,y\in B_R\left(u_i\right)$
\begin{equation} \nonumber
\left(H_{u_i,R}\right)_{x,y}=\left\{ \begin{array}{cc} \left(J_{l_i,R}\right)_{n_i(x),n_i(y)} & \textrm{if } x=y \\
\frac{1}{\sqrt{d-1}}\left(J_{l_i,R}\right)_{n_i(x),n_i(y)} & \textrm{if } x\neq y \\ \end{array} \right.
\end{equation}
Similarly, each term $\left(L_{v_0',R}\right)_{x,y}$ corresponds to a term $\left(J^{(r)}_{0,R}\right)_{\widetilde{n}(x),\widetilde{n}(y)}$
(note that $L_{v_0',R}$ is spherically symmetric around $v_R'$, but not around $v_0'$). Moreover, by the construction of $L$, the maps
$\widetilde{n}$ and $n_i$ are related by $\widetilde{n}\left(f_i(x)\right)=n_i(x)$ (for $i\in\mathbb{N},\ x\in B_R\left(u_i\right)$).
It follows, using \eqref{eq:Jrightlim}, that for any $R\in\mathbb{N}$
\[
\left\Vert I_{f_i}H_{u_i,R}I_{f_i}^{-1} - L_{v_0',R}\right\Vert<d^R\left\Vert
J_{l_i,R}-J_{0,R}^{(r)}\right\Vert\underset{i\to\infty}\longrightarrow 0.
\]
Thus by Proposition \ref{prop:RlimOnRegT}, $L$ is an $\mathcal{R}$-limit of $H$.

The spherical decomposition of $L^{\left(j\right)}$ produces a direct sum of half-line Jacobi matrices,
\begin{equation} \label{eq:Ljdecomp}
L^{\left(j\right)}\cong\bigoplus_{i=0}^{\infty}\left(\oplus L_{i}^{\left(j\right)}\right),
\end{equation}
Now, from each approximate eigenfunction of $J^{\left(r\right)}$
we can produce approximate eigenfunctions of $L_{i}^{\left(j\right)}$ above, for any $j-i$ large enough: assume $g$ is an approximate
eigenfunction of $J^{\left(r\right)}$, satisfying $\left\Vert J^{\left(r\right)}g-\lambda g\right\Vert <\varepsilon$, since
$\| g \|_2=1$ we can take $N$ large enough s.t.\ $\left\Vert g|_{\mathbb{Z}\backslash\left(-N,N\right)}\right\Vert
<\varepsilon$.
For any $m\in\mathbb{Z}$ define $h_{m}\in\ell^{2}\left(\mathbb{N}\right)$
by
\begin{equation}
h_{m}(n)=\begin{cases}
g(n-m-1) & \left|n-m\right|<N\\
0 & \left|n-m\right|\geq N
\end{cases},\nonumber
\end{equation}
then for any $j,i\in\mathbb{N}$ s.t. $j-i>N$,
\[
\left(L_{i}^{\left(j\right)}h_{j-i}\right)\left(k\right)=\left(J^{\left(r\right)}g|_{[-N,N]}\right)\left(k-j+i-1\right)
\]
for every $k\in\mathbb{N}\cap\left(-N-j+i-1,N-j+i-1\right)$. Thus,
\begin{eqnarray*}
\left\Vert L_{i}^{\left(j\right)}h_{j-i}-\lambda h_{j-i}\right\Vert  & \leq & \left\Vert J^{\left(r\right)}g|_{[-N,N]}-\lambda
g|_{[-N,N]}\right\Vert +\varepsilon
\end{eqnarray*}
\[
\leq\left\Vert J^{\left(r\right)}g-\lambda g\right\Vert +\left(\left\Vert J^{\left(r\right)}\right\Vert +\left|\lambda\right|\right)\left\Vert
g|_{\mathbb{Z}\backslash[-N,N]}\right\Vert +\varepsilon<C\cdot\varepsilon.
\]
By the unitary equivalence \eqref{eq:Ljdecomp}, an approximate eigenfuction of some $L_{i}^{\left(j\right)}$ will correspond to an approximate
eigenfunction of $L^{\left(j\right)}$, with the same eigenvalue.
Moreover, since the infinite matrix $L^{(j)}_i$ depends only on $j-i$, the same function is an approximate eigenfunction of $L^{(j)}$ for any
$j$ large enough. Thus, using the strong convergence $L^{(j)}\to L$ we get an approximate eigenfunction of $L$. Therefore
$\sigma\left(J^{(r)}\right)\subseteq \sigma\left(L\right)$.

We now turn to the second case indicated in Proposition \ref{prop:PropA}, and assume $J_{\left(s\right)}$
is a strong limit of the sequence $\left\{ J_{k}\right\} _{k=1}^{\infty}$.
Any such $J_{\left(s\right)}$ will appear as  the restriction to the half line $\ell^{2}\left(\mathbb{N}\right)$ of some right limit
$J^{\left(r\right)}$ which corresponds, as above, to an $\mathcal{R}$-limit, $L$.
Thus $J_{(s)}$ is contained in the set of matrices $\left\{ L_{i}^{\left(j\right)}\right\} _{i,j=0}^{\infty}$
above. Thus, again, any approximate eigenfunction of $J_{\left(s\right)}$
corresponds to an approximate eigenfunction of some $L_{i}^{\left(j\right)}$,
and thus also of $L$.  Hence $\sigma\left(J_{(s)}\right)\subseteq \sigma\left(L\right)$.
\end{proof}


\end{document}